\documentclass[draft, a4paper]{amsart}
\usepackage[T1]{fontenc}
\usepackage[latin1]{inputenc}
\usepackage{amsmath}
\usepackage{amsfonts}
\usepackage{amssymb}  
\usepackage{mathrsfs}  
\usepackage[all, cmtip]{xy}
\usepackage{bm}
\usepackage{pifont}
\usepackage{caption}
\usepackage{array}
\usepackage{setspace}
\usepackage{graphicx}

   \DeclareMathOperator{\N}{\mathbb{N}}
   \DeclareMathOperator{\R}{\mathbb{R}}
   \DeclareMathOperator{\C}{\mathbb{C}}
   
   \DeclareMathOperator{\supp}{supp}
   
   \DeclareMathOperator{\om}{\mathcal{O}_M}
   \DeclareMathOperator{\oc}{\mathcal{O}_C}

   \DeclareMathOperator{\ptp}{\hat{\otimes}_{\pi}}
   \newcommand{\oms}{\mathcal{O}_{M}'}
   \newcommand{\ocs}{\mathcal{O}_{C}'}
   \newcommand{\XX}{\mathscr X}
   \DeclareMathOperator{\Proj}{Proj}

\newtheorem{proposition}{Proposition}

\begin{document}


\keywords{derived projective limit functor $\cdot$ slowly increasing functions $\cdot$ rapidly decreasing distributions $\cdot$ sequence space representations}
\subjclass[2000]{Primary 46M18 $\cdot$ 46F05 ; Secondary 46A08 $\cdot$ 46A45}
\title[Bornologicity of $\mathcal{O}_M$]{A new proof for the bornologicity of the space of slowly increasing functions}
\author{Julian Larcher}
\address{Institut für Mathematik, Universität Innsbruck, A-6020 Innsbruck, Austria}
\email{julian.larcher@uibk.ac.at}
\author{Jochen Wengenroth}
\address{Universität Trier, FB IV -- Mathematik, 54286 Trier, Germany
}
\email{wengenroth@uni-trier.de}

\begin{abstract}
A. Grothendieck proved at the end of his thesis that the space $\om$ of slowly increasing functions and the space $\ocs$ of rapidly decreasing distributions are bornological. 
Grothendieck's proof relies on the isomorphy of these spaces to a sequence space and we present the first proof that does not utilize this fact by using homological methods 
and, in particular, the derived projective limit functor.
\end{abstract}

\maketitle


\section{Introduction and notation}

In \cite[p.~243]{SCHW} L. Schwartz introduced the space of multipliers of temperate distributions, i.e., 
the space of slowly increasing functions 
\[\om=\{f\in\mathcal{C^\infty}(\R^d)\,;\, \forall \alpha\in\N_0^d\,\exists N\in\N:\, \langle x\rangle^{-N}\partial^\alpha f\in L^\infty\},\]
where $\mathcal{C^\infty}(\R^d)$ 
is the space of complex valued, infinitely differentiable functions on $\R^d$, $\langle x \rangle=1+|x|^2$, $\partial^\alpha$ is the partial derivative, and 
$L^\infty$ is the Lebesgue space of bounded functions. 
The dual $\oms$ of $\om$ is the space of very rapidly decreasing distributions.
\par 
Schwartz also introduced the space of convolutors of temperate distributions, i.e., the space $\ocs$ of rapidly decreasing distributions, 
which is the dual of the space
\[\oc=\{f\in\mathcal{C}^\infty(\R^d)\,;\, \exists N \,\forall \alpha\in\N_0^d:\, \langle x\rangle^{-N}\partial^\alpha f\in L^\infty\}\]
of very slowly increasing functions. These spaces are related as in the diagram
\begin{large}
\begin{center}
\begin{tabular}{c c c c c c c c c c c c c}
$\oc$ & $\subseteq$ &$\om$  \\
    \rotatebox[origin=c]{90}{$\cong$}  & & \rotatebox[origin=c]{90}{$\cong$} \\
$\oms$& $\subseteq$&$\ocs$
\end{tabular}
\end{center}
\end{large}
\vspace{8pt}
where in both cases the Fourier transform can be taken as the isomorphism. \\\par 
It is comparatively easy to see that the four spaces are nuclear and semi-reflexive, that $\om$ and $\ocs$ are complete and that $\oc$ and $\oms$ are 
(LF)-spaces and hence bornological. But the completeness of $\oc$ and $\oms$ and the bornologicity of $\om$ and $\ocs$ are not trivial (which was even 
asserted by Grothendieck, \cite[Chap.~II, p.~130]{GRO}). Since the dual of a bornological space is complete and the dual of a complete nuclear space is bornological, 
these two problems are equivalent (for the definitions of these topological properties and relations between them see \cite[Section 424]{EDM}). \par
Grothendieck proved that $\om$ is bornological by showing that it is isomorphic to a complemented subspace of the sequence space 
$s\ptp s'$ \cite[Chap.~II, Lemme 18, p.\,132]{GRO} and verified ``directly'' that the space $s\ptp s'$ is bornological \cite[Chap.II, Prop.~15, p.\,125, Cor.~2, p.~128]{GRO}. 
We will find out more about this isomorphy in Section \ref{sec2} and also give a homological proof of the bornologicity of $s\ptp s'$. \\\par
In \cite{KUC85}, J. Ku\v{c}era claimed to have presented a new (and simple) proof for the main properties of the space $\om$. 
That Ku\v{c}era's proof contains severe mistakes and that it is based on incorrect propositions is clarified in \cite{LAR12}, where also the lack of a proof of 
the bornologicity of $\om$, that does not use the isomorphy $\om\cong s\ptp s'$, is pointed out. In Section \ref{sec3} we will give such a proof.


\section{Projective limits and the space $s\ptp s'$}\label{sec2}

Since quotients (and, in particular, complemented subspaces) of bornological spaces are bornological, it was sufficient for Grothendieck to prove that $\om$ is isomorphic to a 
complemented subspace of $s\ptp s'$, where $s$ is the space of rapidly decreasing sequences
\[ s=\{(x_j)_{j\in\N}\in \C^{\N}\,;\, \forall k: \sup_{j\in\N} j^k |x_j|<\infty\} \]
and $s'$ is its dual, the space of slowly increasing sequences 
\[s'=\{(x_j)_{j\in\N}\in \C^{\N}\,;\, \exists k: \sup_{j\in\N} j^{-k} |x_j|<\infty\}.\]
By $s{\otimes}_\pi s'$ we denote the completed projective tensor product of these spaces. E.g., by \cite[Remark~1, p.~321]{BAR12a}, this space $s\ptp s'$ is canonically isomorphic to
\[
s\ptp s'\cong \{x\in\C^{\N\times\N}\,;\, \forall n \,\exists N :\,  \sup_{i,j}i^n j^{-N} |x_{i,j}|<\infty\}.
\]
In \cite{VAL}, M.\ Valdivia proved that $\om$ is even isomorphic to $s\ptp s'$ itself which answered a question 
posed in \cite[Chap.~II, p.~134]{GRO}. C.\ Bargetz used this fact, the bornologicity of $s\ptp s'$, and methods of the theory 
of topological tensor products to obtain the isomorphy $\oc\cong s\hat{\otimes}_\iota s'$ \cite[Prop.~1, p.~318]{BAR12a}.
\\\par
The descriptions of the spaces $\om$ and $s\hat{\otimes}_\pi s'$ already indicate how they can be written as projective limits of 
LB-spaces (countable inductive limits of Banach spaces)
\begin{gather}\label{om}
\om= \bigcap_{n\in\N} X_{n}=\bigcap_{n\in\N} \bigcup_{N\in\N} X_{n,N},\\\label{s}
s\hat{\otimes}_\pi s'=\bigcap_{n\in\N} Y_{n}=\bigcap_{n\in\N} \bigcup_{N\in\N} Y_{n,N},
\end{gather}
where $X_{n,N}$ and $Y_{n,N}$ are the Banach spaces
\begin{gather*}
X_{n,N}=\{f\in\mathcal{C}^n(\R^d)\,;\,\Vert f\Vert_{n,N}=\sup_{x\in\R^d,|\alpha|\leq n}\langle x \rangle^{-N}| \partial^\alpha f(x)|<\infty\},\\
Y_{n,N}=\{x\in\C^{\N\times\N}\,;\, \| x\|_{n,N}=\sup_{i,j}i^n j^{-N} |x_{i,j}|<\infty
\}.
\end{gather*}
These representations as projective limits of LB-spaces are not only natural but also extremely useful since there are very good criteria for checking bornologicity.
They are related to the derived projective limit functor $\text{Proj}^1 \XX$ 
(which can be defined as the cokernel of the map $\prod X_n \to \prod X_n$, $(x_n)_n\mapsto (x_n-\varrho_{n+1}^n(x_{n+1}))_n$ 
where $\varrho_m^n$ are the connecting maps of the projective spectrum $\XX$, in our cases, $\varrho_m^n$ are just inclusions).
Indeed, an unbublished theorem of D.\ Vogt (his proof reproduced in \cite[Th.\ 3.3.4]{WLN}) says that $\Proj \XX$ is bornological
whenever $\Proj^1 \XX=0$. Moreover, there is a variety of evaluable conditions ensuring $\Proj^1\XX=0$. We are going to apply the following results of
Palamodov-Retakh \cite{Pal} and the second named author, respectively:
\begin{quote}
  A specturm $\XX$ of LB-spaces satisfies $\Proj^1\XX=0$ if and only if there are Banach discs $D_n$ in $X_n$ with $\varrho_m^n(D_m)\subseteq D_n$ and
  \[ \forall\,  n\in\N \;\exists\, m\ge n\; \forall \, k\ge m:\; \varrho_m^n(X_m)\subseteq \varrho_k^n(X_k) + D_n.\]
\end{quote}
The requirement $\varrho_m^n(D_m)\subseteq D_n$ is sometimes very easy to fulfil but in many cases it is very inconvenient. It can be omitted if either all steps $X_n$
are LS-spaces (i.e., the inclusions $X_{n,N}\hookrightarrow X_{n,N+1}$ are compact) or if a slightly stronger condition of Palamodov-Retakh type is required. Denoting by 
$\varrho_\infty^n: \Proj\XX \to X_n$ the obvious map we have:
\begin{quote}
  A specturm $\XX$ of LB-spaces satisfies $\Proj^1\XX=0$ if and only if, for every $n\in\N$, there are a Banach discs $D_n$ in $X_n$ and $m\ge n$ with 
  \[ \varrho_m^n(X_m)\subseteq \varrho_\infty^n(\Proj \XX) + D_n.\]
\end{quote}
We refer to \cite{WLN} for the proofs of these characterization and much more information about derived functors.
Typically, the decompositions required in conditions of Retakh-Palamodov type are quite easy to produce in the case of spaces of sequences (or matrices) since one can write
$x= \chi\, x + (1-\chi)x$ where $\chi$ is the indicator function of a suitably chosen set. We want to exemplify this by giving a very short proof for the bornologicity of
$s\hat{\otimes}_\pi s'$ 
(which is similar to 
Vogt's proof of $\text{Ext}^1(s,s)=0$ \cite[Lemma 2.1, p.~359]{VOG83}).

\begin{proposition}\label{prop:s}
The space $s\hat{\otimes}_\pi s'$ is bornological.
\end{proposition}
\begin{proof}
We keep the notation $s\hat{\otimes}_\pi s'\cong\bigcap_{n\in\N} Y_{n}=\bigcap_{n\in\N} \bigcup_{N\in\N} Y_{n,N}$ from above
and we will verify the Palamodov-Retakh condition for the unit balls $D_n$ of $Y_{n,0}$ which trivially satisfy $D_{n+1}\subseteq D_n$. 
For $n\in\N$ we take $m=n+1$ and fix $x\in Y_n$ as well as $k\ge n+1$. Since $x\in Y_{m,M}$ for some $M\in \N$ we have
\[
  \|x\|_{m,M} =\sup_{i,j} i^m j^{-N} |x_{i,j}|=c <\infty.
\]
We set $y_{i,j}=x_{i,j}$ if $i < cj^M$ and $y_{i,j}=0$ else, as well as $z=x-y$.
For $i<cj^m$ we have $z_{i,j}=0$ and for $i\ge cj^M$ we estimate
\[
 i^nj^{-0} |z_{i,j}| = i^m j^{-M} |z_{i,j}|\, j^M/i \le \|x\|_{m,M}/c=1
\]
which proves $z\in D_n$. It remains to show $y\in Y_{k,K}$ for $K$ sufficienly large. Indeed, for $K=M(k-m+1)$ we have $y_{i,j}=0$ if $i\ge cj^M$ and if $i< cj^M$ we estimate
\[
  i^kj^{-K}|y_{i,j}| =i^m j^{-M} |y_{i,j}| i^{k-m}j^{M-K} \le \|x\|_{m,M} c^{k-m} j^{(k-m)M + M- K} = c^{k-m+1}.
\]
This proves $\|y\|_{k,K}<\infty$, as required.
\end{proof}


\section{The new proof}\label{sec3}

Now we want to prove $\text{Proj}^1\XX=0$ for the spectrum $\XX=(X_n)_{n\in\N}$ in (\ref{om}) in order 
to obtain that $\om$ is bornological. 
Splitting up a given function $f\in X_m$  as $f=\chi f+ (1-\chi) f$ with a cut-off function $\chi$ (as in the proof of Proposition \ref{prop:s})
does not work in this case. But we will see how $f$ can be ``split up'' in the following proof of Grothendieck's result.

\begin{proposition} The space $\om$ is bornological.
\end{proposition}
\begin{proof}
To obtain $\text{Proj}^1\XX=0$ we will show
\begin{gather}\label{hom}
\forall n\,\exists m,N:\,\,X_{m}\subseteq \om + B_{n,N}
\end{gather}  where $B_{n,N}$ is the unit ball of $X_{n,N}$. 
This condition means that we have to approximate every $f\in X_m$ with respect to the norm $\|\cdot\|_{n,N}$ by elements of $\om$.
To achieve such an approximation we use a kernel $K\in\om(\R^d\times\R^d)$ satisfying
\begin{align*}
& K\ge 0,\; \int_{\R^d} K(t,x)\,dt=1\text{ for all }x\in\R^d,\, \text{ and} \\
&\supp K(\cdot,x)\subseteq \prod\limits_{j=1}^{d}[x_j,x_j+\varepsilon \langle x \rangle^{-\mu}]=:A_x \text{ for all }x\in\R^d
\end{align*}
where we will see later how $\varepsilon$ and $\mu$ have to be chosen in dependence on $f\in X_m$. We can obtain such a kernel by defining 
\[
K(t,x)=\varepsilon^{-d}\langle x \rangle^{\mu d} \varphi(\varepsilon^{-1}\langle x \rangle^\mu(t-x))
\]
for a positive test function $\varphi\in\mathcal{C}^\infty(\R^d)$ with support in $[0,1]^d$ and $\int_{\R^d} \varphi(t) dt=1$ 
(the conditions above can be checked easily and $K\in\om$ since every derivative of $K$ can be estimated by a polynomial). \par
We start with the one-dimensional case $d=1$ where we can take $m=n+1$ and $N=0$. So let $f\in X_{n+1,M}$ for some $M\in \N$. 
We want to find $g\in \om$ such that $f-g\in B_{n,0}$. At first we set
\[
g_n(x)=\int_{\R}f^{(n)}(t) K(t,x)\,dt
\]
and show that this is a good approximation to $f^{(n)}$. Since for $l\in\N_0$ 
\[
\left| g_n^{(l)}(x) \right|=\left|\int_{A_x}f^{(n)}(t) \partial_{x}^{l} K(t,x)\,dt\right|\leq\int_{A_x} |P(t)|\, |Q(t,x)| \,dt\leq |R(x)|
\] for some polynomials $P$, $Q$, $R$, the function $g_n$ is contained in $\om$. Furthermore we can estimate in virtue of Taylor's formula \begin{gather*}
\left| f^{(n)}(t)-f^{(n)}(x)\right|\leq |t-x| \langle \xi(t,x) \rangle^M \Vert f\Vert_{n+1,M}
\end{gather*} with a point $\xi(t,x)$ between $t$ and $x$. For $\varepsilon$ small enough the inequality $\langle \xi(t,x)\rangle\leq 2\langle x \rangle$ holds for every $x\in\R$ and $t\in A_x$. We obtain
\begin{equation}\label{est}
\begin{aligned}
|g_n(x)-f^{(n)}(x)|&=\left|\int_{\R}\left(f^{(n)}(t)-f^{(n)}(x)\right) K(t,x)\,
dt\right|\\
&\leq\int_{A_x}\left|f^{(n)}(t)-f^{(n)}(x)\right| K(t,x)\,dt\\
&\leq \int_{A_x}|t-x|\langle \xi(t,x) \rangle^M \Vert f\Vert_{n+1,M}K(t,x)\,dt\\
&\leq\varepsilon \,2^M\, \langle x\rangle^{M-\mu} \Vert f\Vert_{n+1,M}\int_{A_x} K(t,x)\,dt\\
&=\varepsilon\,2^M\,\langle x\rangle^{M-\mu}\Vert f\Vert_{n+1,M}.
\end{aligned}
\end{equation}
Now if 
\begin{gather*}T:\om(\R)\rightarrow\om(\R), h\mapsto \left(x\mapsto \int_{0}^{x}h(t)\,dt\right),
\end{gather*}
we can set
\[g(x)=\sum_{j=0}^{n-1}\frac{f^{(j)}(0)}{j !}x^j +(T^n g_n)(x).\] 
Then $g\in\om$ and since
\[(T^n f^{(n)})(x)=f(x)-\sum_{j=0}^{n-1}\frac{f^{(j)}(0)}{j !}x^j,\]
integrating (\ref{est}) (the integral starting at 0) yields
\[
| g^{(l)}(x)-f^{(l)}(x)|\leq 1,\,x\in\R^d,l\leq n
\]
for $\varepsilon$ small enough and $\mu$ large enough. Hence $g-f\in B_{n,0}$ and the proof is complete for the one-dimensional case. 
\par
Now we will prove the two-dimensional case  $d=2$. We set $m=2n+1$ and $N=n-1$ in (\ref{hom}). So let $f\in X_{2n+1,M}$ for some $M$. With the help of a kernel 
$K\in\om(\R^2\times \R^2)$ like above, we set \[g_n(x)=\int_{\R^2}\partial^{(n,n)}f(t)K(t,x)\,dt\] in order to approximate 
$\partial^{(n,n)}f$ by $g_n$. Similar to the one-dimensional case we have 
\begin{gather*}
\left| \partial^{(n,n)}f(t)-\partial^{(n,n)}f(x)\right|\leq c\cdot|t-x| \langle \xi(t,x) \rangle^M \Vert f\Vert_{2n+1,M}
\end{gather*} and $\langle \xi(t,x)\rangle\leq 2\langle x \rangle$ for $t\in A_x$ and $\varepsilon$ small enough and thus
\begin{equation}\label{est2}
\begin{aligned}
|g_n(x)-\partial^{(n,n)}f(x)|&\leq c \int_{A_x}|t-x|\langle \xi(t,x) \rangle^M \Vert f\Vert_{2n+1,M}K(t,x)\,dt\\
&\leq \tilde{c}\,\varepsilon\,\langle x\rangle^{M-\mu}\Vert f\Vert_{2n+1,M}.
\end{aligned}
\end{equation}
Let us denote $T_j$ the integral with respect to the $j$-th component (the integral starting at 0). Applying $T_{1}\circ T_{2}$ $n$-times to $\partial^{(n,n)}f(x)$ yields 
\begin{multline*}
(T_{1}^{n}T_{2}^{n}f)(x)=\\f(x)+ \sum_{\alpha<(n,n)}\partial^\alpha f (0,0) \frac{x^\alpha}{\alpha !} - 
\sum_{j=0}^{n-1}\partial^{(j,0)} f(0,x_2) \frac{x_1^j}{j !}- \sum_{j=0}^{n-1}\partial^{(0,j)} f(x_1,0) \frac{x_2^j}{j !}.
\end{multline*}
As in the one-dimensional case we can choose 
$g^1_0,\ldots g^1_{n-1},g^2_0,\ldots,g^2_{n-1}\in\om(\R)$ such that $\Vert g^1_j-\partial^{(0,j)} f(\cdot,0)\Vert_{n,0}
\leq \varepsilon$ and $\Vert g^2_j-\partial^{(j,0)} f(\cdot,0)\Vert_{n,0}\leq \varepsilon$. Defining
\[
g(x)=(T_{1}^{n}T_{2}^{n})g_n(x)- \sum_{\alpha<(n,n)}\partial^\alpha f (0,0) \frac{x^\alpha}{\alpha !} + 
\sum_{j=0}^{n-1}g^2_j(x_2) \frac{x_1^j}{j !}+ \sum_{j=0}^{n-1}g^1_j(x_1) \frac{x_2^j}{j !}
\]
and applying $T_{1}^{n}T_{2}^{n}$ to (\ref{est2}) yields
\begin{multline*}
|g(x)-f(x)|\leq \\\varepsilon +\sum_{j=0}^{n-1}\left( \left|g^1_j(x_1)-\partial^{(0,j)} f(x_1,0)\right| \frac{|x_2|^j}{j !} 
+\left| g^2_j(x_2)-\partial^{(j,0)} f(0,x_2)\right|\frac{|x_1|^j}{j !}\right)
\end{multline*}
for $\mu$ large enough which implies
\begin{align*}
|g(x)-f(x)|\leq \varepsilon +\varepsilon \sum_{j=0}^{n-1}\frac{|x_2|^j}{j !} +\varepsilon \sum_{j=0}^{n-1}\frac{|x_1|^j}{j !}\leq \varepsilon \,c \,\langle x \rangle^{n-1}
\end{align*}
for some $c>1$. Since similar estimates also hold for $|\partial^\alpha g(x)-\partial^\alpha f(x)|$, $|\alpha|\leq n $, we obtain $g-f \in B_{n,n-1}$ and the proof is complete for $d=2$.
\par
The general case $d\in\N$ is very similar. Inductively we want to show $$X_{dn+1}\subseteq \om+B_{n,(d-1)(n-1)}$$ and start by 
approximating $\partial^{(n,\ldots,n)}f$ by $g_n(x):=\int_{R^d}\partial^{(n,\ldots,n)}f(t)K(t,x)\,dt$. 
Then we integrate the estimate of $g_n-\partial^{(n,\ldots,n)}f$ $n$-times with respect to each component.
The integral $T_1^n\cdots T_d^n \partial^{(n,\ldots,n)}f$ contains $f$ as a summand and terms that are the 
product of a derivative of $f$ that only depends on less than $d$ components and a polynomial in less than $d$ components 
with exponents less than $n$. But we can estimate the functions that only depend on less than $d$ variables by the induction 
hypothesis and hence we can obtain $g\in\om$ with $g-f\in B_{n,(d-1)(n-1)}$.
\end{proof}



\end{document}